\documentclass[12pt,thmsa, reqno]{amsart}

\let\oldtocsection=\tocsection

\let\oldtocsubsection=\tocsubsection

\renewcommand{\tocsection}[2]{\hspace{0em}\oldtocsection{#1}{#2}}
\renewcommand{\tocsubsection}[2]{\hspace{2em}\oldtocsubsection{#1}{#2}}

\usepackage{setspace}

\usepackage{amssymb, euscript,  mathrsfs}
\usepackage[dvips]{graphics}
\usepackage[title]{appendix}

\input{amssym.def}

\input{amssym.tex}

\usepackage{lineno}

\usepackage{tikz}
\usepackage[utf8]{inputenc}

\usetikzlibrary{matrix,arrows,decorations.pathmorphing}

\usepackage[all,cmtip]{xy}

\usepackage{amsmath}
\let\oldAA\AA
\renewcommand{\AA}{\text{\normalfont\oldAA}}

\def\Cal{\mathcal}

\def\C{{\Cal C}}

\def\S{{\Cal S}}

\def\bbr{{\Bbb R}}

\def\bbs{{\Bbb S}}

\def\rn{\bbr^n}

\def\part{\partial}
\def\intl{\int\limits}

\def\Lam{\Lambda}

\def\Gam{\Gamma}

\def\a{\alpha}

\def\Del{\Delta}

\def\vp{\varphi}

\def\Gam{\Gamma}
\def\Lam{\Lambda}
\def\sig{\sigma}
\def\lam{\lambda}

\def\e{\varepsilon}

\def\chi{{\bf 1}}

\def\snm1{\bbs^{n-1}}

\def\intl{\int\limits}

\def\Cs{\mathscr{C}}

\def\Cs{\mathscr{C c 1234}}

\def\cd{\stackrel{*}{\C}\!{}_{m, k}^\lam}
\def\sd{\stackrel{*}{\S}\!{}_{m, k}^\lam}
\def\cd0{\stackrel{*}{\C}\!{}_{m, k}^\lam}
\def\sd0{\stackrel{*}{\S}\!{}_{m, k}^\lam}

\def\ncd0{\stackrel{*}{\Cs}\!{}_{m, k}^\lam}

\newtheorem{theorem}{Theorem}[section]
\newtheorem{lemma}[theorem]{Lemma}

\theoremstyle{definition}

\theoremstyle{remark}

\numberwithin{equation}{section}

\theoremstyle{corollary}

\numberwithin{equation}{section}

\newcommand{\be}{\begin{equation}}
\newcommand{\ee}{\end{equation}}

\newcommand{\bea}{\begin{eqnarray}}
\newcommand{\eea}{\end{eqnarray}}
\newcommand{\Bea}{\begin{eqnarray*}}
\newcommand{\Eea}{\end{eqnarray*}}

\def\sideremark#1{\ifvmode\leavevmode\fi\vadjust{\vbox to0pt{\vss
 \hbox to 0pt{\hskip\hsize\hskip1em
\vbox{\hsize2cm\tiny\raggedright\pretolerance10000
 \noindent #1\hfill}\hss}\vbox to8pt{\vfil}\vss}}}%

\begin{document}



\title[A Note on   the Sonar Transform]
{A Note on   the Sonar Transform and Related Radon Transforms}

\author{ B. Rubin}

\address{Department of Mathematics, Louisiana State University, Baton Rouge,
Louisiana 70803, USA}
\email{borisr@lsu.edu}

\subjclass[2010]{Primary 44A12; Secondary 47G10}



\keywords{Sonar Transforms, Radon transforms, $L^p$ estimates, inversion formulas.}

\begin{abstract}
The sonar transform in geometric tomography maps 
 functions on the Euclidean half-space to integrals of those  functions over hemispheres centered on the boundary hyperplane.
 We obtain sharp $L^p$-$L^q$ estimates for this  transform and new explicit inversion formulas under minimal assumptions for functions.
 The main results follow from  intriguing connection between the sonar transform, the Radon  transform over paraboloids, and the transversal Radon transform,
 which integrates functions on $\rn$ over hyperplanes, meeting the last coordinate axis.

 \end{abstract}

\maketitle


\section{Introduction}

 We consider an integral  transform over hemispheres in the  half-space
$\rn_+ = \{y=(y', y_n): \, y' \in \bbr^{n-1}, \; y_n > 0\}$
 defined by
\be \label {hemi}
(H \vp)(x',r)=\intl_{S_+ (x',r)} \!\!\vp (y)\, d\sig (y),\qquad x' \in \bbr^{n-1}, \quad r>0,\ee
where  $S_+ (x',r)=\{y\in \rn_+: |y-x'|=r\}$, $ n\ge 2$, and  $d\sig (y)$ denotes the induced surface area measure. This integral operator is  known in tomography as  the  {\it sonar transform} (see, e.g., \cite{Be04,  LQ, QRS}), where
the word {\it sonar} is an acronym for “sound navigation and ranging”.   One of the main problems is reconstruction of $\vp$ if $(H \vp)(x',r)$ is known for all $(x',r)$. This problem amounts to mathematical models in acoustic imaging, seismology, thermoacoustic tomography \cite{BM, BK, Che01, No76, No80} and leads to interesting developments in integral geometry and harmonic analysis. Many researchers contributed to the study of the operators (\ref{hemi}) and related spherical means; see, e.g., \cite{An,  D99, Fa, NR, NRT, NT, Pal00, Pal04}, to mention a few; see also survey articles \cite{KK08, KK15}.

 The purpose of this note is to obtain sharp weighted norm estimates and explicit inversion formulas  for $H \vp$.
 Unlike many other publications on this subject, in which  $\vp$ is  smooth and compactly supported away from the boundary, our consideration is focused on
 arbitrary Lebesgue integrable functions $\vp$.

Our approach relies on intimate connection between $H \vp$ and two other Radon-type transforms
\bea \label {ppar}
(P f)(x) &=& \intl_{\bbr^{n-1}} f(x'-y', x_n  -|y'|^2)\, dy', \\
 \label {bart}
(T\psi)(x)&=&  \intl_{\bbr^{n-1}} \psi (y', x'\cdot y' +x_n)\, dy',\eea
which map functions on $\rn$ to functions on $\rn$.
The first one can be called   the {\it  parabolic Radon transform} and 
resembles integration of $f$ over the shifted  paraboloid
\[\pi_x= \pi_0 +x, \qquad \pi_0= \{y=(y', y_n):  y_n= - |y'|^2\}.\]
However, $(P f)(x)$ differs from the usual surface integral
\be \label {papar}
\intl_{\pi_x} \!f(y)\, d\sig (y)=\intl_{\bbr^{n-1}} f(x'-y', x_n  -|y'|^2) \, (1+4|y'|^2)^{1/2} dy'\ee
 by the Jacobian factor, which is suppressed in our consideration.
 The connection between $H$ and $P$ was indicated in \cite{Be09, BK, NR, No80}, though the use of this connection was different from ours.

The  Radon  transform  (\ref{bart})  differs from the most common one, as, e.g., in \cite[Chapter I, Section 2]{H11}; see (\ref{RRT}) below.
 It  resembles  affine Radon transforms in Gelfand’s school \cite{GGG, GGV} and the parametric
Radon transforms in \cite{E}.   Following Strichartz \cite{Str},
who used it in the context of the Heisenberg group  (see also  \cite{Ru12}, \cite [Section 4.13] {Ru15}),  we call
 (\ref{bart})  the {\it  transversal Radon transform}, taking into account that
it  integrates functions over only  those hyperplanes
 which meet the last coordinate axis. Both transforms (\ref{ppar}) and (\ref{bart}) were applied by Christ \cite{Chr} to obtain
  optimal constants in the corresponding $L^p$-$L^q$ inequalities.

\subsection{Main results} We shall prove the following statement.

\begin{theorem}\label{trs} For the hemispherical transform  (\ref{hemi}), the  inequality
\[\Bigg (\,\intl_{\bbr^{n-1}} \!\!\Big[ \intl_0^\infty \!|(H\vp) (x', r)|^s \, r^{1-s}\, dr\Big ]^{q/s}\!\! dx' \Bigg )^{1/q}\!\! \le c_{p, q, s} \Bigg (\,\intl_{\rn_+} \! |\vp (y)|^p \, y_n^{1-p}\, dy\Bigg )^{1/p} \]
 holds
if and only if
\be\label {isp} 1 \le p < n/(n-1), \qquad q=p^{\prime},\qquad 1/s = 1-n/p'.\ee
In particular, for  $q=s=n+1$ and  $p=(n+1)/n$,
\[\Bigg (\,\intl_{\bbr^{n-1}} \!\! \intl_0^\infty \!|(H\vp) (x', r)|^{n+1} \frac{dr dx'}{r^n} \Bigg )^{1/(n+1)}\!\!\!\! \!\le c_{n} \Bigg (\,\intl_{\rn_+} \! |\vp (y)|^{(n+1)/n} \frac{dy}{y_n^{1/n}}\,\Bigg )^{n/(n+1)}\!\!\!. \]
\end{theorem}

Here, as usual, $p'$ is the dual exponent, defined by $1/p +1/p'$=1. An important feature of this theorem  is that it gives precise information about the behavior  of  $H\vp$ both globally and near the boundary, depending on the  behavior  of $\vp$.

We also obtain explicit inversion formulas for the operators $H$ and $P$, which follow from the  known results for $T$; see Theorems \ref {loza},  \ref{invek2},
 \ref{loam},  \ref{invek2m}.

\noindent{\bf Plan of the Paper.}  Section 2 contains notation and auxiliary facts about the transversal Radon transform $T$.
In Section 3 we establish basic connections between the operators $H$, $P$, and $T$,  and obtain inversion formulas for $P$.  Section 4 contains the proof of Theorem \ref{trs} and  inversion formulas for $H$.

\section{Preliminaries}

\subsection{Notation}  In the following, $x\!=\!(x_1, \ldots, x_{n-1}, x_n)\!=\!(x', x_n) \!\in \!\rn$; $\Delta = \partial^2/\partial x_1^2 +\ldots +\partial^2/\partial x_n^2$ is
 the Laplace operator.
The notation $C(\bbr^n)$,  $C^\infty (\bbr^n)$,
and $L^p (\bbr^n)$ for function spaces is standard;  $C_0 (\bbr^n)=\{f\in C(\bbr^n):
\lim\limits_{|x|\to\infty} f(x) = 0\}$; $C_c^\infty (\bbr^n)$ is the space of compactly supported infinitely differentiable functions on $\rn$. All integrals are meant as Lebesgue integrals. We say that  an integral under consideration
 exists in the Lebesgue sense if it is finite when the integrand is replaced by its absolute value.  A letter $c$ stands for an inessential positive constant which may vary at each  occurrence.
Given a certain real-valued expression
$X$, and a complex number $\lam$, we set $(X)_\pm^\lam= |X|^\lam$ if $\pm X>0$ and $(X)_{\pm}^\lam=0$, otherwise.

\subsection{The transversal Radon transform}\label{llza}

We recall that the  most familiar form of the Radon transform  \cite{H11} is
\be \label {RRT}  (Rf)(\theta,t)=\intl_{\theta^\perp} f(y+t\theta)\,dy, \qquad (\theta,t) \in  S^{n-1} \times \bbr, \ee
where $S^{n-1}$ is the unit sphere in $\rn$ and $\theta^\perp$ is the hyperplane through the origin orthogonal to $\theta$. Setting $\theta\!=\!(\theta_1,
\ldots,\theta_n)\!=\!(\theta', \theta_n )$ and
\be (\Lam\vp)(\theta, t)=\vp \left (-\frac{\theta'}{\theta_n},\frac{t}{\theta_n}\right),
\ee
we have
 \be \label{45}(Rf)(\theta,
t)=|\theta_n |^{-1} (\Lam T f)(\theta, t), \qquad  \theta_n \neq
0. \ee
  This formula can be inverted and enables one to reformulate all known facts for $R$ in terms of $T$;  see \cite [Section 4.13.5] {Ru15} and \cite [Section 5.2]{Chr} for details.
  For example, consider the mixed norm
  \be\label{daz}
\|F\|_{q, s} \equiv \Bigg (\, \intl_{\bbr^{n-1}} \Big [\intl_\bbr |F (x',x_n)|^s
dx_n\Big]^{q/s}\!\! dx'\Bigg)^{1/q}.\ee
Then the celebrated Oberlin-Stein theorem \cite {OS} for the operator $R$  yields  the following statement for $T$.
\begin{theorem}\label{tos} For $n\ge 2$,
 the inequality
$$
\| T\psi\|_{q, s}  \le c_{p, q, s}  \| \psi \|_p$$ holds
if and only if $p$, $q$, and $s$ satisfy (\ref{isp}).
In particular, for  $q=s=n+1$ and  $p=(n+1)/n$,
\be\label{edxaz}
||T\psi||_{n+1}\le c_n ||\psi||_{(n+1)/n}.\ee
\end{theorem}

The proof of Theorem \ref{tos} can be found in  \cite [Theorems 1.4, 8.2] {Ru12}, \cite[Theorem 4.149] {Ru15}); see also \cite [Lemma 2.2]{Chr}.
The necessity of the bounds for $p$, $q$, and $s$ are inherited from those for $R$. They can also be checked straightforward,
 using the scaling argument. For example, let $\lam =(\lam_1, \lam_2)$, $\lam_1 >0$, $\lam_2 >0$. We denote
\[
(A_\lam\psi)(x)=\psi (\lam_1 x', \lam_2 x_n), \qquad (B_\lam \Psi)(x)=\lam_1^{1-n} \Psi\left (\frac{\lam_2}{\lam_1}\, x', \lam_2 x_n\right ), \]
so that $T A_\lam \psi = B_\lam T \psi$. Then
\[ ||A_\lam \psi||_p=\lam_1^{(1-n)/p}\lam_2^{-1/p} ||\psi||_p, \quad ||B_\lam \Psi||_q=\lam_1^{1-n+(n-1)/q}\lam_2^{-n/q} ||\Psi||_q.\]
If $|| T \psi||_q \le c \,||\psi||_p$ is true for all $\psi\in L^p$, then it is true for $A_\lam \psi$, that is, $||T A_\lam \psi||_{q}\le c \,||A_\lam \psi||_{p}$, or
$||B_\lam T \psi||_{q}\le c \,||A_\lam \psi||_{p}$.
 The latter is equivalent to
\[
\lam_1^{1-n+(n-1)/q}\lam_2^{-n/q} ||T\psi||_q   \le c \,\lam_1^{(1-n)/p}\lam_2^{-1/p} ||\psi||_p.\]
Letting $\lam_1 $ and $\lam_2 $ tend to zero and to infinity, we conclude that the last inequality is possible only if
$p=(n+1)/n$ and $ q=n+1$.

A variety of inversion formulas for the transversal Radon transform
can be found in \cite [Section 4] {Ru12}; see also \cite [Section 4.13] {Ru15}. For example, the following statement holds.

\begin {theorem}\label {loa}
A function $\psi\in L^p (\rn)$, $1\le p < n/(n-1)$, can be reconstructed from $\Psi=T\psi$ by the formula
\be\label{for1} \psi(x) \! = \!
\frac{1}{d_{n,\ell}(n\!-\!1)}\intl_{\bbr^n} \!\! \frac{(\Del^\ell_y
g)(x)}{|y|^{2n-1}}\, dy, \ee where
\[
d_{n,\ell}(n\!-\!1) \! = \! \intl_{\bbr^n}\!\!
\frac{(1 \! - \! e^{iy_1})^\ell}{|y|^{2n-1}}\, dy
\quad \text{($y_1$ is the first coordinate of $y $)}, \]
\be\label {ggg} g (x)=(2\pi)^{1-n} \intl_{\bbr^{n-1}} \Psi (y', x_n -y' \cdot x')\, \frac{dy'}{(1+|y'|^2)^{(n-1)/2}},\ee
\[ (\Del^\ell_y g)(x)=\sum_{j=0}^\ell {\ell \choose j} (-1)^j g(x-jy).\]
 Here $\ell=n-1$ if $n$ is even, and $\ell>n-1$ is arbitrary if $n$ is odd. The integral in (\ref {for1}) is understood as a limit $\lim\limits_{\e \to
0}\int_{|y|>\e}$. This limit exists in the $L^p$-norm and in the
a.e. sense. For $\psi\in C_0 (\rn)\cap L^p (\rn)$,  it exists in the $\sup$-norm.
\end{theorem}

An alternative inversion formula can be obtained in terms of powers of the minus Laplacian operator.
We follow \cite[Theorem 4.5]{Ru12} (or \cite[Theorem 4.132]{Ru15}), assuming for simplicity $\psi$ to be smooth and
compactly supported (this assumption is weakened in   \cite[Subsection 4.1.2]{Ru12} in terms of Lipschitz functions with prescribed decay at infinity).

\begin{theorem}  \label{ine2} Let  $\psi\in C_c^\infty (\rn)$,  $\Psi=T\psi$, and let $g$ be defined by (\ref{ggg}).
If $n$ is odd,  then
 \be \label{pf1}\psi(x)= (-\Del)^{(n-1)/2}\,g(x).\ee
 If $n$ is even,  then
  \be \label{pf2} \psi(x)\!=\! c_n \intl_{\bbr^n}\!
\frac{(-\Del)^{n/2-1}g (x)\!-\!(-\Del)^{{n/2-1}}g (x\!-\!y)}{|y|^{n+1}} \,
dy , \ee
where $c_n=\Gam ((n\!+\!1)/2)/\pi^{(n+1)/2}$ and  $\int_{\bbr^n}=\lim\limits_{\e \to 0}\int_{|y|>\e}$
uniformly in $x \in \bbr^n$.
\end{theorem}

\section {Basic Connections}

In this section we establish basic connections
 between the hemispherical transform $H$, the parabolic transform $P$, and
the transversal Radon transform $T$.

\subsection {Connection between $H$ and $P$}
By  the classical Calculus,
\bea
&&(H\vp)(x',r)=r\intl_{|y'|<r} \vp(x'-y', \sqrt {r^2 -|y'|^2})\, \frac{dy'}{\sqrt {r^2 -|y'|^2}}\nonumber\\
&&
=r\intl_{|y'|<r} f (x'-y', r^2 -|y'|^2)\, dy', \qquad f(z)= z_n^{-1/2} \vp (z', \sqrt {z_n}); \quad z\in \rn_+.\nonumber\eea
Replace  $r^2$ by $x_n>0$ to get
$(H\vp)(x',\sqrt {x_n})=\sqrt {x_n}\,(P_0 f)(x)$, where
\be \label {par}
(P_0 f)(x) = \intl_{|y'|<\sqrt {x_n}} f(x'-y', x_n  -|y'|^2)\, dy', \qquad x \in \rn_+.\ee
Setting
\bea
\label{aaa}(A\vp)(z) &\equiv& (A\vp)(z', z_n)= z_n^{-1/2} \vp (z', \sqrt {z_n}),  \quad z\in \rn_+,\\
\label{aaa1}(A^{-1} \Phi)(x)&\equiv&  (A^{-1} \Phi)(x', x_n)=x_n \Phi(x', x_n^2), \quad x\in \rn_+,\eea
so that $A^{-1} A\vp=\vp$, $AA^{-1}\Phi=\Phi$,
 we  obtain
\be \label {conn} (H\vp)(x', r)= (A^{-1} P_0 A\vp)(x', r).\ee
The next step is to extend (\ref{par})  to functions on  $\rn$.
Given a function $f^+$ on $\rn_+$, we denote by $e_{-}f^+$  its extension by zero onto the entire space $\rn$, and let $r_+$ be the restriction map,
which assigns to a function on $\rn$ its restriction onto $\rn_+$.  Clearly,
\be \label {bar}
(P_0 f^+)(x) = (r_+ P e_{-} f^+)(x), \qquad x\in \rn_+.\ee
Hence (\ref{conn}) gives  the following statement.
\begin{lemma} \label{ii7y} The equality
\be\label {iaiw}
H\vp=A^{-1}r_+ P e_{-} A\vp\ee
holds provided that either side of it exists in the Lebesgue sense.
\end{lemma}

\subsection {Connection between  $P$ and  $T$.}\label {99jh}

This connection can be found in \cite [Lemma 2.3]{Chr}. In view of its importance and for the sake of completeness, 
we present it in detail, using our notation. For $x \in \rn$, let
\be \label {bars}
(B_1f)(x)=f(x', x_n \!-\!|x'|^2), \; (B_2 F)(x)=F(2x',  x_n \!-\!|x'|^2).\ee
The corresponding inverse maps have the form
\be \label {bars1} (B_1^{-1}u)(x) \! = \!u(x', x_n +|x'|^2), \qquad (B_2^{-1} v)(x)\!=\!v \left (\frac{x'}{2},  x_n +\frac{|x'|^2}{4}\right ).\ee
One can readily see that
\be\label {iar} ||B_1f||_p =||f||_p, \qquad \| B_2 F \|_{q, s} = 2^{(1-n)/q} \,\|F \|_{q, s}.\ee

\begin{lemma}\label {swa} The equality
\be \label {bts}  Pf=B_2TB_1 f,\ee
holds provided that either side of it exists in the Lebesgue sense.
\end{lemma}
\begin{proof} We write the left-hand side as
\[
(P f)(x) \!= \!\intl_{\bbr^{n-1}}\!\! f(y', x_n \! -\!|x'\!-\!y'|^2)\, dy'\!=\!\intl_{\bbr^{n-1}} \!\!f(y', x_n \! -\!|x'|^2 \!-\!|y'|^2 +2 x'\cdot y')\, dy'.\]
Hence
\[
(B_2^{-1}P f)(x)= (P f)\!\left (\frac{x'}{2},  x_n +\frac{|x'|^2}{4}\right )=\intl_{\bbr^{n-1}}\! \!  f (y', x_n -|y'|^2+ x'\cdot y')\, dy'.\]
On the other hand,
\[
(TB_1f)(x)=  \intl_{\bbr^{n-1}} \!\!  (B_1f)(y', x'\cdot y' +x_n)\, dy'=\intl_{\bbr^{n-1}}\! \!  f (y', x_n -|y'|^2+ x'\cdot y')\, dy',\]
as above. This gives the result.
\end{proof}

\begin{theorem}\label{tosq}
For $n\ge 2$,
 the  inequality
$$
\| Pf\|_{q, s} \le c_{p, q, s}  \| f \|_p$$ holds
if and only if $p$, $q$, and $s$ satisfy (\ref{isp}).
In particular, for  $q=s=n+1$ and $p=(n+1)/n$,
\be\label{edxaz1}
||Pf||_{n+1}\le c_n \,||f||_{(n+1)/n}.\ee
\end{theorem}

Theorem \ref{tosq} follows immediately from
Lemma \ref{swa} and (\ref{iar}). The inequality (\ref{edxaz1}) can also be found in \cite [formula (1.10)]{Chr}.

Another important consequence of Lemma \ref{swa} is a series of explicit inversion formulas for the parabolic Radon transform $P$.
 Specifically, let  $Pf=F$. Then, by (\ref{bts}), $f= B_1^{-1} T^{-1} B_2^{-1}F$, or
\[
f=B_1^{-1} \psi, \qquad \psi =T^{-1} \Psi, \qquad \Psi= B_2^{-1}F.\]
By (\ref{bars1}),
\[(B_1^{-1} \psi)(x)  =\psi(x', x_n +|x'|^2), \qquad
(B_2^{-1} F)(x)  = F\left (\frac{x'}{2},  x_n +\frac{|x'|^2}{4}\right ),\]
and $\psi =T^{-1} \Psi$ can be computed using the results from Subsection \ref{llza}. Hence
 Theorem \ref{loa}  gives the following statement.

\begin {theorem}\label {loza}
A function $f\in L^p (\rn)$, $1\le p < n/(n-1)$, can be reconstructed from $F=Pf$ by the formula
\be\label{fvv}
f(x) \! = \!
\frac{1}{d_{n,\ell}(n\!-\!1)}\intl_{\bbr^n} \!\! \frac{(\Del^\ell_y
g)(x',x_n +|x'|^2)}{|y|^{2n-1}}\, dy,\ee
where the right-hand side has the same meaning as in (\ref{for1}) but with
\be\label{fvv1s}
g (x)=\pi^{1-n} \intl_{\bbr^{n-1}} F (z', x_n - 2z' \cdot x' + |z'|^2 )\, \frac{dz'}{(1+ 4|z'|^2)^{(n-1)/2}}.\ee
\end{theorem}

Another inversion result follows from  Theorem \ref {ine2}.

\begin{theorem}  \label{invek2} Let  $f\in C_c^\infty (\rn)$,  $F=Pf$, and let $g$ be defined by (\ref{fvv1s}).
If $n$ is odd,  then
 \be \label{f1}f(x)= B_1^{-1}\left [ (-\Del)^{(n-1)/2}\,g(x)\right ].\ee
 If $n$ is even,  then
  \be \label{f2} f(x)\!=\!B_1^{-1}\left [ c_n \int_{\bbr^n}\!
\frac{(-\Del)^{n/2-1}g (x)\!-\!(-\Del)^{{n/2-1}}g (x\!-\!y)}{|y|^{n+1}} \,
dy\right ], \ee
where $c_n=\Gam ((n\!+\!1)/2)/\pi^{(n+1)/2}$, $\int_{\bbr^n}=\lim\limits_{\e \to 0}\int_{|y|>\e}$
uniformly in $x \in \bbr^n$.
\end{theorem}

\subsection {Connection between  $H$ and  $T$.}\label {99jhc}
Combining Lemmas \ref{ii7y} and  \ref{swa}, we obtain
\[ H\vp=Q_2TQ_1 \vp, \qquad Q_1=B_1 e_{-} A, \qquad  Q_2= A^{-1}r_+ B_2.\]
Explicit expressions for $Q_1$, $Q_2$, and their inverses can be easily obtained using
 (\ref {bars}), (\ref {bars1}), (\ref {aaa}), and (\ref {aaa1}). We have
\bea\label{fsw2} (Q_1 \vp)(x) \! \!&=& \! \! (x_n  \!- \!|x'|^2)_+^{-1/2} \vp (x', \sqrt {x_n \! - \!|x'|^2}), \quad x \in \rn;\qquad \\
\label{fsw3}(Q_2 f)(x',r)\! \! &=&\! \!  r f (2x',  r^2 -|x'|^2),\quad x' \in \bbr^{n-1}, \quad r>0.\eea
Similarly,
\bea\label{fsw4}
(Q_1^{-1} \psi)(y)\! \! \!&=& \! \!\! y_n \psi (y',  y_n^2 +|y'|^2), \quad y \in \rn_+;\\
\label{fsw5}(Q_2^{-1} \Phi)(x)\! \!\!&=&\!\! \!\left (x_n \!+\!\frac{|x'|^2}{4}\right )_+^{-1/2}\!\! \!\Phi \!\left (\frac{x'}{2}, \sqrt{x_n \!+\!\frac{|x'|^2}{4}}\,\right  ), \quad x \in \rn.\qquad \eea

Thus we get one more important connection.

\begin{lemma}\label {swar} The equality
\be \label {btsz}  H\vp=Q_2TQ_1 \vp,\ee
holds provided that either side of it exists in the Lebesgue sense.
\end{lemma}

Note that
\bea
||Q_1 \vp||_p^p &=& \intl_{\bbr^{n-1}} dx' \intl_{|x'|^2}^\infty  (x_n -|x'|^2)^{-p/2} \, |\vp (x', \sqrt {x_n -|x'|^2})|^p \,d x_n\nonumber\\
\label {reda}&=& 2\intl_{\rn_+}  |\vp (y)|^p \, y_n^{1-p}\, dy\eea
and
\bea
||Q_2^{-1} \Phi||_{q, s}^q &=&  \intl_{\bbr^{n-1}} \Big [\intl_\bbr |(Q_2^{-1}\Phi) (x',x_n)|^s
dx_n\Big]^{q/s}\!\! dx'\nonumber\\
&=&\intl_{\bbr^{n-1}} \Bigg [\intl_\bbr \left (x_n +\frac{|x'|^2}{4}\right )_+^{-s/2} \Bigg |\Phi \!\left (\frac{x'}{2}, \sqrt{x_n +\frac{|x'|^2}{4}}\right )\Bigg |^s dx_n \Bigg ]^{q/s}\nonumber\\
\label {reda1}&=&2^{n-1 +q/s} \intl_{\bbr^{n-1}} \Big [ \intl_0^\infty |\Phi (x', r)|^s \, r^{1-s}\, dr\Big ]^{q/s}\, dx'.\eea

\section{Proof of the Main Results}

\subsection{ Proof of Theorem \ref{trs}} Denote
\[
||\vp||_p^{\sim}= \Bigg (\,\intl_{\rn_+} \! |\vp (y)|^p \, y_n^{1-p}\, dy\Bigg )^{1/p},  \]
\[
||\Phi||^{\sim}_{q,s}=\Bigg (\, \intl_{\bbr^{n-1}} \!\!\Big[ \intl_0^\infty \!|\Phi (x', r)|^s \, r^{1-s}\, dr\Big ]^{q/s}\!\! dx' \Bigg )^{1/q}.\]
Then (\ref{reda}) and (\ref{reda1}) yield
\[
||Q_1 \vp||_p= 2^{1/p} ||\vp||^{\sim}_p, \qquad ||Q_2^{-1} \Phi||_{q, s}= 2^{(n-1)/q +1/s} ||\Phi||^{\sim}_{q,s}.\]
Our aim is to show that $||H\vp||^{\sim}_{q,s} \le c\, ||\vp||^{\sim}_p$ for some constant $c$.  Indeed,
by Lemma \ref{swar} and Theorem \ref{tos},
\bea
||H\vp||^{\sim}_{q,s}&=& 2^{(1-n)/q -1/s} ||Q_2^{-1} H\vp||_{q, s}=  2^{(1-n)/q -1/s} ||TQ_1 \vp||_{q, s}\nonumber\\
&\le& c\, ||Q_1 \vp||_p =  2^{1/p} c\, ||\vp||^{\sim}_p,\nonumber\eea
as desired.

\subsection{Inversion formulas}
Let us proceed to inversion of $H\vp$.
A formal inversion formula $\vp= (Q_1^{-1}T^{-1} Q_2^{-1})H\vp$ can be made precise using explicit formulas for $Q_1^{-1}$ and $Q_2^{-1}$ combined with Theorems \ref{loa} and \ref{ine2}. Specifically,
let $H\vp=\Phi$. Setting $\Psi =Q_2^{-1} \Phi$  in (\ref{ggg}) and using  (\ref{fsw5}), in slightly different notation we obtain
\bea
g (x)&=&\pi^{1-n} \intl_{\bbr^{n-1}} (x_n - 2z' \cdot x' +  |z'|^2)_+^{-1/2}\nonumber\\
\label{fvv1}&\times& \Phi \left(z', \sqrt {x_n - 2z' \cdot x' +  |z'|^2} \right )\, \frac{dz'}{(1+ 4|z'|^2)^{(n-1)/2}}.\qquad \eea

Then Theorems \ref{loa} yields the following result.

\begin {theorem}\label {loam} A function $\vp$ satisfying
\[
\intl_{\rn_+}  |\vp (y)|^p \, y_n^{1-p}\, dy<\infty, \qquad 1 \le p < n/(n-1),\]
 can be reconstructed from $\Phi=H\vp$ by the formula
\be\label{fvvm}
\vp(y) \! = \!
\frac{y_n}{d_{n,\ell}(n\!-\!1)}\intl_{\bbr^n} \!\! \frac{(\Del^\ell_t
g)(y', y_n^2 +|y'|^2)}{|t|^{2n-1}}\, dt,\ee
where $g$ is defined by (\ref{fvv1}) and the right-hand side has the same meaning as in (\ref{for1}).
\end{theorem}

In a similar way,   Theorem \ref{ine2} implies the following statement.

\begin{theorem}  \label{invek2m} Let  $H\vp=\Phi$, where $\vp$ is a $C^\infty$ function with compact support in $\rn_+$. Suppose that
 $g$ is defined by (\ref{fvv1}).
If $n$ is odd,  then
 \be \label{f1a}  \vp(y)= Q_1^{-1}\left [ (-\Del)^{(n-1)/2}\,g(x)\right ](y).\ee
 If $n$ is even,  then
  \be \label{f2a} \vp(y)\!=\!Q_1^{-1}\left [ c_n \int_{\bbr^n}\!
\frac{(-\Del)^{n/2-1}g (x)\!-\!(-\Del)^{{n/2-1}}g (x\!-\!z)}{|z|^{n+1}} \,
dz\right ](y), \ee
where     $c_n=\Gam ((n\!+\!1)/2)/\pi^{(n+1)/2}$, $\int_{\bbr^n}=\lim\limits_{\e \to 0}\int_{|z|>\e}$
uniformly in $x \in \bbr^n$, and $Q_1^{-1}$ acts in the $x$-variable.
\end{theorem}

\vskip 0.3 truecm

\subsection{Some generalizations} It is natural to take one step further and proceed from the  Radon transforms  (\ref{ppar}) and (\ref{bart})
to the corresponding  one-sided fractional integrals
 \bea\label{poaxz}
(P_{\pm}^{\,\a} f)(x) &=&\frac{1}{\Gam (\a)} \intl_{\bbr^n}(y_n -|y'|^2)_{\pm} ^{\a -1} \,f(x-y)\, dy,\\
\label {99jh1}
 (T_{\pm}^\a f)(x)&=& \frac{1}{\Gam (\a)}\intl_{\rn} (x_n -y_n)_{\pm}^{\a -1} f(y', y_n + x' \cdot y')\, dy,\eea
 which yield $P$ and $T$, respectively, as  $\a\to 0$. 
 These generalizations are of interest on their own right.

The integrals (\ref{99jh1}) seem to be new.
The $L^p$-$L^q$ estimates of the localized modifications of (\ref{poaxz}),  containing a cut-off function under the sign of integration, were briefly discussed by
 Littman \cite {Litt}.  The non-localized integrals  (\ref{99jh1}) differ from those in  \cite {Litt}, though they have some  features in common.
 We plan to address this topic in another paper.

\end{document}